\RequirePackage[T1]{fontenc}
\documentclass[12pt]{amsart}

\usepackage[height=8.85in,width=6.45in]{geometry}

\usepackage[whole]{bxcjkjatype}
\usepackage[utf8]{inputenc}
\usepackage{amsmath}
\usepackage{amssymb}
\usepackage{mathtools}
\numberwithin{equation}{section}
\usepackage{slashed}
\usepackage{braket}
\usepackage[svgnames]{xcolor}
\usepackage[colorlinks,citecolor=DarkGreen,linkcolor=FireBrick,urlcolor=FireBrick,linktocpage,breaklinks=true]{hyperref}
\usepackage{cite}
\usepackage[pdftex]{graphicx}
\usepackage{tikz}
\usepackage{tikz-cd}
\usepackage{times}
\usepackage{courier}
\usepackage{bm}
\usepackage{subfig}

\usepackage{xcolor}
\usepackage{mdframed}

\usepackage{mathrsfs}

\renewenvironment{figure}[1][]{
  \begin{originalfigure}[#1]
    \begin{mdframed}[linecolor=black!0,backgroundcolor=black!1]
}{
    \end{mdframed}
  \end{originalfigure}
}
\def\Arf{\mathop{\mathrm{Arf}}}
\def\ABK{\mathop{\mathrm{ABK}}}
\def\Irr{\mathop{\mathrm{Irr}}}
\def\qdim{\mathop{\mathrm{qdim}}}

\def\bC{\mathbb{C}}
\def\bH{\mathbb{H}}
\def\bR{\mathbb{R}}
\def\bZ{\mathbb{Z}}
\def\cS{\mathcal{S}}
\def\cQ{\mathcal{Q}}
\def\BW{\mathrm{BW}}
\def\Hom{\mathrm{Hom}}
\def\tr{\mathrm{tr}}
\def\Cl{\mathrm{Cl}}
\def\B{\mathrm{B}}
\def\Z#1{\mathbb{Z}/#1\mathbb{Z}}

\newtheorem{theorem}[equation]{Theorem}
\newtheorem{lemma}[equation]{Lemma}

\begin{document}

\title[The super Frobenius-Schur indicator]{The super Frobenius-Schur indicator and \\
finite group gauge theories on pin$^-$ surfaces}
\author{Takumi Ichikawa}
\author{Yuji Tachikawa}
\date{February 25, 2020.\relax\quad \textit{Preprint number:} IPMU-20-0013}
\address{Kavli Institute for the Physics and Mathematics of the Universe \textsc{(wpi)},
  The University of Tokyo,
  5-1-5 Kashiwanoha, Kashiwa, Chiba, 277-8583,
  Japan
}
\email{takumi.ichikawa@ipmu.jp, yuji.tachikawa@ipmu.jp}
\thanks{The research of Yuji Tachikawa is  supported by in part supported  by WPI Initiative, MEXT, Japan at IPMU, the University of Tokyo,
and in part by JSPS KAKENHI Grant-in-Aid (Wakate-A), No.17H04837 
and JSPS KAKENHI Grant-in-Aid (Kiban-S), No.16H06335.
The authors thank Arun Debray for careful reading and instructive comments on an earlier draft of this paper.
The authors also thank John Murray for helpful email exchanges after v1 appeared on the arXiv, informing them the relevance of \cite{Gow} to our work.}

\begin{abstract}
It is well-known that the value of the Frobenius-Schur indicator $|G|^{-1} \sum_{g\in G} \chi(g^2)=\pm1$
of a real irreducible representation of a finite group $G$ determines which of the two types of real representations it belongs to, i.e.~whether it is strictly real or quaternionic.
We study the extension to the case when a homomorphism $\varphi:G\to \Z2$ gives the group algebra $\bC[G]$ the structure of a superalgebra.
Namely, we construct of a super version of the Frobenius-Schur indicator whose value for a real irreducible super representation is an eighth root of unity,
distinguishing which of the eight types of irreducible real super representations described in \cite{Wall} it belongs to.
We also discuss its significance in the context of two-dimensional finite-group gauge theories on  pin$^-$ surfaces.
\end{abstract}
\maketitle

\tableofcontents

\section{Introduction}
It is a classic result of Frobenius and Schur \cite{FS} that for a finite group $G$ and its irreducible representation $\rho$ over $\bC$ whose character we denote by $\chi$, the Frobenius-Schur indicator defined by \begin{equation}
S(\rho):=\frac{1}{|G|} \sum_{g\in G} \chi(g^2)
\end{equation} 
takes values in $\{+1,0,-1\}$, depending on whether $\rho$ is  strictly real, complex or quaternionic. 
A slight reformulation can be given in terms of the simple summand $A_\rho$ of the group algebra $\bC[G]$.
%We have $S(\rho)=0$ when $\rho\not\simeq \bar \rho$.
When $\rho\simeq\bar\rho$, $A_\rho$ has a conjugate-linear involution $*$, whose fixed point is a central simple algebra over $\bR$.
From Wedderburn's theorem, this is a matrix algebra over a division algebra $D$ over $\bR$ whose center is $\bR$, which is either $\bR$ or $\bH$.
These algebras $\bR$, $\bH$ form the Brauer group $\B(\bR)=\{\bR,\bH\}\simeq \Z2$ of $\bR$.
The value $S(\rho)=\pm1$ then determines to which element of $\B(\bR)$ the simple summand $A_\rho$ corresponding to the representation $\rho$ belongs.

In this paper we generalize these statements to the group superalgebras.\footnote{%
Various other generalizations of the Frobenius-Schur indicator were studied in the literature.
Firstly, we can consider generalized Frobenius-Schur indicators for finite groups, see e.g.~\cite{BumpGinzburg} and \cite[Sec.~11]{georgieva2018klein}.
Secondly, we can extend the concept of the Frobenius-Schur indicator to fusion categories, see e.g.~\cite[Sec.~3]{fuchs2001category},
which is further extended to fermionic fusion categories in e.g.~\cite[Sec.~3.3]{Bultinck_2017}.
This last generalization should agree with ours when specialized to fermionic fusion categories arising from our superalgebra.
It would be interesting to check this.
}
Let us first recall the Brauer-Wall group $\BW(F)$ of a field $F$ \cite{Wall,NotesOnSpinors}.
A division superalgebra $D=D_0\oplus D_1$ 
is a superalgebra such that every nonzero element in $D_0$ or $D_1$ is invertible.
Consider a simple superalgebra $A=A_0\oplus A_1$ over $F$,
with its unique irreducible supermodule $\rho$.
The super version of Schur's lemma says that $D_A:=\Hom_A(\rho,\rho)$ 
is a division superalgebra over $F$.
Simple superalgebras $A$ such that the center of the even part $(D_A)_0$ is $F$ itself
are called central simple,
and central simple superalgebras under the equivalence relation $A\sim A' \Leftrightarrow D_A \simeq D_{A'}$
form a finite group $\BW(F)$ under tensor product.
This is the Brauer-Wall group of $F$.
It was determined in \cite{Wall} that $\BW(\bC)=\Z2$ and $\BW(\bR)=\Z8$.
%A central simple superalgebra $A$ over $\bC$ whose unique irreducible supermodule is $\rho$ corresponds to $0\in \Z2$ if $\rho$ is irreducible as a ungraded module of $A$,
%and $1\in \Z2$ otherwise.
The Clifford algebras over $\bC$ and $\bR$, with the corresponding periodicity $2$ and $8$, cover all ten cases. 
%We recall the detailed description of $\BW(\bR)=\Z8$ later.

Let us now specialize to the case of superalgebras associated to a group,
and formulate our main theorem.
Let $G$ be a finite group, 
$\varphi:G\to \Z2$ a homomorphism or equivalently a one-cocycle,
and $\alpha:G\times G \to \Z2$ a two-cocycle. 
$\varphi$ gives a decomposition of $G$ as $G=G_0\sqcup G_1$.
We define $\bR[G]_{\varphi,\alpha}$ to be the superalgebra over $\bR$ generated by elements $e_g$ for $g\in G$ such that $e_g$ is even or odd depending on $\varphi(g)=0$ or $1$,
and $e_g e_h= (-1)^{\alpha(g,h)} e_{gh}$.
We denote its complexification by $A:=\bC[G]_{\varphi,\alpha}$.
This algebra is equipped with a conjugate linear involution $a\mapsto a^*$, inherited from the complex conjugation of $\bC$.
These superalgebras are semisimple, as can be seen e.g.~by using the unitary trick.
For a supermodule $\rho$ defined on $V=V_0\oplus V_1$, we define its character as $\chi(g):=\tr_V \rho(e_g)$.

Let $(\rho,V)$ be an irreducible supermodule of $A=\bC[G]_{\varphi,\alpha}$.
The corresponding subsuperalgebra $A_\rho$ of $A$ is central simple,
and we denote its class $[A_\rho]\in \BW(\bC)$ by $q(\rho)=0,1$.
It is known that $\rho$ is irreducible as an ungraded module if $q(\rho)=0$,
and $\rho|_{V_0}$ is irreducible as a module of $\bC[G_0]_\alpha$ if $q(\rho)=1$.

We denote by $\bar\rho$ the complex conjugate supermodule of $\rho$.
When $\bar\rho \simeq \rho$, $\rho$ is called real.
Otherwise $\rho$ is called complex.
Real irreducible supermodules of $\bC[G]_{\varphi,\alpha}$ can be classified by $\BW(\bR)=\Z8$,
just as  irreducible supermodules of $\bR[G]_{\varphi,\alpha}$ can be.

We show the following:
\begin{theorem}\label{main}
The super Frobenius-Schur indicator of an irreducible supermodule $\rho$, defined by $$
\cS(\rho) := \frac{1}{\sqrt{2}^{q(\rho)}|G|} \sum_{g\in G} \sqrt{-1}^{\varphi(g)} (-1)^{\alpha(g,g)} \chi(g^2),
$$
is either zero or an eighth root of unity.
It is zero if  $\rho$ is complex
and  non-zero if $\rho\simeq \bar \rho$.
In the latter case, its value $\cS(\rho)\in \{ x^8=1 \mid x\in \bC\}$
determines to which element of $\BW(\bR)\simeq \Z8$ the supermodule belongs.
\end{theorem}

The rest of the paper consists of two sections.
In Section~\ref{sec:proof},
we provide a proof of the main theorem.
In Section~\ref{sec:background},
we describe how the super Frobenius-Schur indicator appears in the context of two-dimensional finite-group gauge theories on pin$^-$ surfaces.
%This last section will be written in a less rigorous style.

\section{Proof of the main theorem}
\label{sec:proof}

We start by quoting the structure of central simple complex or real superalgebras $A=A_0\oplus A_1$ and the corresponding Brauer-Wall groups $\BW(\bC)=\Z2$ or $\BW(\bR)=\Z8$, as described in \cite{Wall,NotesOnSpinors}.
For simplicity of the presentation, we assume that any superalgebra $A$ contains a regular element in $A_1$.

Over $\bC$, any central simple superalgebra $A=A_0\oplus A_1$ is either of the form \begin{equation}
M(n) = \{    
\begin{pmatrix}
a &b \\
c & d
\end{pmatrix} 
\} 
\end{equation} or  \begin{equation}
Q(n)=\{
\begin{pmatrix}
a & b \\
b & a
\end{pmatrix} 
\}
\end{equation}
where $a,b,c,d$ are $n\times n$ matrices.
$a$ and $d$ provide even elements and $b$ and $c$ provide odd elements.
$M(n)$ belongs to $0\in \BW(\bC)$ and $Q(n)$ belongs to $1\in \BW(\bC)$.
We denote the type in $\BW(\bC)$ by $q\in \BW(\bC)$.
We note that $Q(1)$ is the Clifford algebra $\Cl(1)$.

More abstractly, any central simple superalgebra over $\bC$
either admits a nonzero graded central element $u\in A_0$
or a nonzero central element $u\in A_1$,
depending on whether $q=0$ or $1$.
Here, an element $x$ is called graded central if $yx=xy$ for all $y\in A_0$ and $yx=-xy$ for all $y\in A_1$.
The special element $u$ is determined up to a scalar multiplication.
We note that $A$ or $A_0$ is simple as an ungraded algebra when $q=0$ or $q=1$, respectively.

By direct computation, one finds that 
\begin{equation}
\begin{aligned}
M(m)\hat\otimes M(n)&=M(mn),&
M(m)\hat\otimes Q(n)&=Q(mn),\\
Q(m)\hat\otimes M(n)&=Q(mn),&
Q(m)\hat\otimes Q(n)&=M(mn)\oplus M(mn)
\end{aligned}
\label{tensor}
\end{equation}
where $\hat\otimes$ denotes the super tensor product of two superalgebras.

An irreducible supermodule $(\rho,V)$ of the semisimple superalgebra $A=\bC[G]_{\varphi,\alpha}$,
determines a simple summand $A_\rho$ of $A$.
This allows us to associate an element $q\in \BW(\bC)$ to the supermodule $\rho$.
We note that $\rho$ is irreducible as an ungraded module when $q=0$,
and $\rho|_{V_0}$ is irreducible as an ungraded module of $\bC[G_0]_{\alpha}$ when $q=1$.
We recall that $G_0$ is the inverse image of $0\in \Z2$ under $\varphi:G\to \Z2$.

Let us take another finite group $H$,
a homomorphism $\varphi':H\to \Z2$,
and a $\Z2$-valued two-cocycle $\alpha'$ of $H$.
Let us consider a supermodule $\rho$ of $\bC[G]_{\varphi,\alpha}$
and a supermodule $\sigma$ of $\bC[H]_{\varphi',\alpha'}$.
\begin{lemma}\label{grouptensor}
$\rho\hat\otimes\sigma$ is a supermodule of $\bC[G\times H]_{\varphi^\circ,\alpha^\circ}$
where $$
\varphi^\circ=\varphi+\varphi',\qquad
\alpha^\circ=\alpha+\alpha'+\varphi \varphi'.
$$
Here, on the right hand side, $\varphi$, $\varphi'$, $\alpha$, $\alpha'$ are appropriately pulled back from $G$ and $H$ to $G\times H$.
Then, to take the product of $\varphi$ and $\varphi'$, we regard them as one-cocycles so that the result is a two-cocycle, i.e., $(\varphi \varphi')(g,h):=\varphi(g)\varphi'(h)$.
\end{lemma}
\begin{proof}
Direct computation.
\end{proof}

Let us assume that $\rho$ and $\sigma$ are irreducible.
According to \eqref{tensor},
$\rho\hat\otimes\sigma$ is irreducible or a direct sum of two isomorphic irreducible supermodules.
We denote by $\rho\sigma$ the isomorphism class of the irreducible supermodule(s) contained in $\rho\hat\otimes\sigma$.
\begin{lemma}\label{product}
The super Frobenius-Schur indicator is multiplicative in the sense that $$
\cS(\rho\sigma)=\cS(\rho)\cS(\sigma).
$$
\end{lemma}
\begin{proof}
Direct computation.
\end{proof}

Let us further assume that $(\rho,V)$ is real, i.e.~$\rho\simeq \bar\rho$.
In this case the irreducible representations are classified into eight types in the following manner:
\begin{itemize}
\item We first ask whether the corresponding $q\in \BW(\bC)$ is 0 or 1. 
This gives the first two-fold division.
\item We then replace the special element $u$ by $cu$ for some $c\in \bC$ so that we have $u=u^*$ and $u^2=\pm1$.
The sign gives the second two-fold division.
\item Finally, when $q=0$, $\rho$ is real irreducible as an ungraded module of $\bC[G]_{\varphi,\alpha}$,
and we  ask whether $\rho$ is strictly real or quaternionic.
Similarly, when $q=1$, $\rho|_{V_0}$ is real irreducible as an ungraded module of $\bC[G_0]_{\alpha}$,
and we  ask whether $\rho|_{V_0}$ is strictly real or quaternionic.
This gives the third two-fold division.
\end{itemize}

We note that when $\rho\simeq\bar\rho$,
the simple summand $A_\rho$ of $\bC[G]_{\varphi,\alpha}$
has a conjugate-linear automorphism $*$,
whose fixed points form a central simple superalgebra over $\bR$.
We recall that central simple superalgebras over $\bR$ 
can be classified in terms of $\BW(\bR)=\Z8$,
whose structure was stated exactly as above \cite[p.195]{Wall}.
We summarize the structure in the following table:
\begin{equation}
\begin{array}{c|ccc||c|ccc}
\BW(\bR)  & q & u^2 &  \rho & \BW(\bR)  & q&  u^2 &  \rho|_{V_0} \\
\hline
0 & 0 & +1 & \bR & 1 & 1 & +1 & \bR \\
2 & 0 & -1 & \bR & 3 & 1 & -1 & \bH\\
4 & 0 & +1 & \bH & 5 & 1 & +1 & \bH\\
6 & 0 & -1 & \bH & 7 & 1 & -1 & \bR
\end{array}.
\label{table}
\end{equation}

Let us now consider some examples. 
\begin{lemma}\label{clifford0}
Let $G=(\Z2)^n$. 
There is a homomorphism $\varphi^{(n)}: G\to \Z2$ and a $\Z2$-valued two-cocycle $\alpha^{(n)}$ on $G$ such that $\bC[G]_{\varphi^{(n)},\alpha^{(n)}}$ is the Clifford algebra $\Cl(n)$.
\end{lemma}
\begin{proof}
Take  $G^\circ=\Z2$, $\varphi^\circ:\Z2\to \Z2$ is the identity, and $\alpha^\circ\equiv 0$.
We have 
$A^\circ:=\bC[G^\circ]_{\varphi^\circ,\alpha^\circ}=\bC[u]$, where $u$ is odd and $u^2=1$.
This equals the Clifford algebra $\Cl(1)$.
We note that $(A^\circ)^{\hat\otimes n}$ is the Clifford algebra $\Cl(n)$.
From the lemma \ref{grouptensor} above, this is a twisted group superalgebra associated to $G=(\Z2)^n$ with a specific choice of $\varphi$ and $\alpha$.
\end{proof}

By comparing  the explicit description of $A^\circ=\Cl(1)$ and the table~\ref{table},
we find that $A^\circ=\Cl(1)$ corresponds to $1\in \BW(\bR)$.
Therefore $\Cl(n)$ corresponds to $n\in\BW(\bR)$.
It has a unique irreducible supermodule.

\begin{lemma}\label{clifford}
For the irreducible supermodule $\rho$ of $\Cl(n)=\bC[(\Z2)^n]_{\varphi^{(n)},\alpha^{(n)}}$ defined above,
the super Frobenius-Schur indicator is given by $$
\cS(\rho)=e^{n\times 2\pi \sqrt{-1}/8 }.
$$
\end{lemma}
\begin{proof}
For $n=1$ this follows from a direct computation.
Then the statement for $n>1$ follows from a repeated application of the lemma \ref{product}.
\end{proof}
We note that the summand $\sqrt{-1}^{\varphi(g)}(-1)^{\alpha(g,g)}:(\Z2)^n\to \{\pm1,\pm\sqrt{-1}\}$
in the definition of the indicator in this case is a  $\Z4$-valued quadratic refinement of the standard non-degenerate pairing on $(\Z2)^n$, 
and that the super Frobenius-Schur indicator is its Arf-Brown invariant itself.

We also need the following lemma:\begin{lemma}\label{projFS}
Let $\bC[G]_\alpha$ be the twisted group algebra generated by $e_g$ for $g\in G$ satisfying $e_ge_h=(-1)^{\alpha(g,h)}e_{g,h}$, where $\alpha(g,h)\in \{0,1\}$ is a two-cocycle.
Its irreducible representation $\rho$ is complex, strictly real, or quaternionic if and only if its Frobenius-Schur indicator $$
S(\rho):= \frac{1}{|G|} \sum_g (-1)^{\alpha(g,g)} \chi(g^2)
$$ 
is $0$, $+1$ or $-1$, respectively.
\end{lemma}
\begin{proof}
We repeat the standard proof which studies whether $\rho\otimes \rho$ contains the identity representation of $G$, and if so, whether it is in $\mathrm{Sym}^2\rho$ or in $\wedge^2 \rho$.
We find that $$
\frac{1}{|G|} \sum_g \tr \rho(e_g e_g) = 0, +1, -1
$$ depending on the three cases. 
Finally we use $\rho(e_g e_g)=(-1)^{\alpha(g,g)} \chi(g^2)$.
\end{proof}

We can now begin the
\begin{proof}[Proof of the main theorem]
The homomorphism $\varphi:G\to \Z2$ gives the decomposition $G=G_0\sqcup G_1$.
When $G_1$ is empty, the theorem reduces to the non-super case, our Lemma \ref{projFS}.
When $\alpha$ is trivial it further reduces to the original theorem of Frobenius-Schur \cite{FS}.
We assume $G_1$ is non-empty in the following. We then have $|G|=2|G_0|$.

We saw in lemma \eqref{product} that the super Frobenius-Schur indicator is multiplicative.
We also saw in lemma \eqref{clifford} that the super Frobenius-Schur indicator of  the irreducible supermodule $\rho^\circ$ of $\Cl(1)=\bC[G^\circ]_{\varphi^\circ,\alpha^\circ}$ is $e^{2\pi\sqrt{-1}/8}$.
Therefore, 
repeating $n$ times  the procedure of tensoring by $\rho^\circ$ and taking the irreducible component 
multiplies the indicator by  $e^{2\pi\sqrt{-1} n /8}$ 
and adds $n$ mod 8 to the class in $\BW(\bR)$.
This means that 
it suffices to prove our statement in  two cases, namely
when $\rho$ is complex and corresponds to $0\in \BW(\bC)$, or
when $\rho$ is real and corresponds to $0\in \BW(\bR)$.
Below, we assume that either of the two is satisfied.

We note that $\cS(\rho)$ can be rewritten as \begin{equation}
\cS(\rho)=\frac{1}{\sqrt{2}^q} \left[
\frac{1}{|G|} \sum_{g\in G_0} (-1)^{\alpha(g,g)} \chi(g^2)
+ 
\frac{\sqrt{-1}}{|G|} (\sum_{g\in G}-\sum_{g\in G_0})  (-1)^{\alpha(g,g)} \chi(g^2) 
\right]
\label{rewrite}
\end{equation}
This means that $\cS(\rho)$ can be computed if the ordinary Frobenius-Schur indicators
of $\rho$ as representations of $G$ and $G_0$ can be computed.

Denote $V=V_0\oplus V_1$ the graded vector space on which $\rho$ is defined.
Due to our assumption, $V$ corresponds to $0\in \BW(\bC)$.
Therefore 
$V$ is irreducible as an ungraded module of $\bC[G]_{\varphi,\alpha}$.
Since $G_1$ is nonempty, 
$\dim V=2\dim V_0$ and $G_0$ is an index-2 subgroup of $G$.
Therefore, 
$V_0$ is an irreducible module of $\bC[G_0]_\alpha$
and $V$ is the induced representation.
Therefore, $V$ is real if $V_0$ is real,
and $V_0$ is  complex if $V$ is complex. 
Therefore, when $\rho$ is complex, both $V$ and $V_0$ are complex,
and   $\cS(\rho)=0$  from \eqref{rewrite}.
Similarly, when $\rho$ is real, both $V$ and $V_0$ are strictly real,
and therefore $\cS(\rho)=1$ from \eqref{rewrite}.
This concludes the proof.
\end{proof}

There is also an alternative proof, based on an earlier work by Roderick Gow \cite{Gow}, which we briefly indicate below\footnote{%
The following three paragraphs were added in v2.  The authors thank J. Murray for the information.
}. 
Given an index-2 subgroup $G_0$ of $G$ and an irreducible representation $\rho_0$,
Gow considered an analogue of the Frobenius-Schur indicator \begin{equation}
\eta(\rho_0):=\frac{1}{|G_0|}\sum_{g\in G\setminus G_0} \chi(g^2),
\end{equation}
which we call the Gow indicator.
In \cite[Lemma 2.1]{Gow}, it was proved that $\eta(\rho_0)=+1, 0, -1$, whose value depends on $S(\rho_0)$ and various properties of the $G$ representation $\mathrm{Ind}_{G_0}^G(\rho_0)$ induced from $\rho_0$ of $G_0$; see also \cite[Sec.~2]{KawanakaMatsuyama}.
The analysis there can be readily generalized to the case when the $\Z2$-valued 2-cocycle $\alpha$ is nontrivial.

For us, we have $\rho_0=\rho|_{V_0}$ and
$\rho=\mathrm{Ind}_{G_0}^G(\rho_0)$.
Then $q=0,1$ depending on whether $\rho$ is irreducible or not as a $G$ representation.
We can then straightforwardly show that  \begin{equation}
\mathcal{S}(\rho) = \frac{1}{\sqrt{2}^q} \left(S(\rho_0) + \sqrt{-1}\eta(\rho_0) \right).
\end{equation}
This allows us to determine our super Frobenius-Schur indicator from the description of the Gow indicator in  \cite[Lemma 2.1]{Gow}, by carefully comparing its proof and our table \eqref{table}.
This gives an alternative proof.

Another related point is that the element of the Brauer-Wall group associated to an irreducible representation $\rho$ of an index-2 subgroup $G_0$ of $G$ was studied in detail in \cite{Turull}.
The discussion there was  for an arbitrary base field, and no explicit formula was given for the base field $\mathbb{R}$ as in our main theorem.

\section{Relation to two-dimensional finite group gauge theories on pin$^-$ surfaces }
\label{sec:background}

According to physicists, quantum gauge theories are `defined' in terms of path integrals.
We start from a $d$-dimensional manifold $M_d$,
and consider the moduli space $\mathcal{M}$ of connections $A$ on principal $G$-bundles on $M_d$.
We pick a function $\mathbb{S}:\mathcal{M}\to \bC/\bZ$ called the action.
The fundamental object of physicists' studies is the partition function \begin{equation}
Z_{G,\mathbb{S}}(M_d) := \int_{\mathcal{M}} e^{2\pi\sqrt{-1} \mathbb{S}} d\mu
\end{equation}
where $d\mu$ is a suitable measure on $\mathcal{M}$.
Of course all this is terribly ill-defined in most of the physically relevant cases when $G$ is a compact Lie group, since $\mathcal{M}$ is heavily infinite-dimensional.

Finite group gauge theories, however, admits a path-integral definition without any problems, since the space $\mathcal{M}$ over which we need to integrate is a finite set,
which is the space of homomorphisms $f:\pi_1(M_d) \to G$ up to conjugation.
The proper measure $d\mu$ to be used was determined e.g.~in \cite{Freed:1991bn},
and the partition function is simply \begin{equation}
Z_{G,\mathbb{S}}(M_d) := \frac{1}{|G|}\sum_{f:\pi_1(M_d)\to G} e^{2\pi\sqrt{-1} \mathbb{S}(f)}
\end{equation}
for a closed and connected manifold $M_d$.
Here we only described the partition function, but they can be promoted to  topological  field theories in the sense of Atiyah-Segal, and further to extended topological field theories.

When the space $M=\Sigma$ is two-dimensional and the action $\mathbb{S}$ is trivial,
this `path integral' was considered independently of physics in \cite{Mednykh},
which says the following.
Let $\Sigma$ be a two-dimensional oriented surface of genus $\gamma$
which is closed and connected.
We then have \begin{equation}
\frac{1}{|G|} \sum_{f: \pi_1(\Sigma)\to G } 1  = \sum_{\rho\in \Irr(G)} \left(\frac{|G|}{ \dim \rho}\right)^{-e(\Sigma)}
\label{Med}
\end{equation}
where the sum on the left hand side is over homomorphisms $f$ from $\pi_1(\Sigma)$ to $G$,
the sum on the right hand side is over isomorphism classes of irreducible representations of $G$,
and $e(\Sigma)=2-2\gamma$ is the Euler number of the surface.
When $\gamma=0$ this reduces to the classic formula $|G|=\sum_{\rho\in\Irr(G)} (\dim\rho)^2$.

\if0
This can be generalized to surfaces with boundaries. 
For example, let $\Sigma$ be a sphere with three boundaries, and consider $G$-bundles whose holonomies around three boundaries are conjugate to $a_{1,2,3}\in G$, respectively.
The corresponding formula is \begin{equation}
\frac{1}{|G|} \sum_{ 
(g_1 a_1 g_1^{-1})  (g_2 a_2 g_2^{-1}) (g_3 a_3 g_3^{-1})=1
} 1 =  \sum_{\rho\in \Irr(G)} \frac{ |G|}{\dim\rho} \chi_\rho(a_1) \chi_\rho(a_2) \chi_\rho(a_3)
\end{equation} where the sum on the left hand side is over $g_{1,2,3}\in G$
and $\chi_\rho$ is the character of the representation $\rho$.
\fi 

For the detailed history of these formulas and  very elementary proofs, see \cite{Snyder}.
This formula is so classic that it appears as exercises of textbooks of finite group theory, e.g.~\cite{Serre}.

The formula \eqref{Med} can be further generalized in many directions. 
One example starts by picking a degree-2 cohomology class $\alpha\in H^2(BG,\bR/\bZ)$.
A homomorphism $f:\pi_1(\Sigma)\to G$ determines a classifying map $f:\Sigma\to BG$ of the corresponding $G$-bundle, which we denoted by the same letter.
We now take the pull-back $f^*(\alpha)$ of $\alpha$ and integrate it against the fundamental class, the result of which we denote by $\int_\Sigma f^*(\alpha)$.
We use it as the action $\mathbb{S}$ and the generalized formula is now given by \begin{equation}
\frac{1}{|G|} \sum_{f: \pi_1(\Sigma)\to G } e^{2\pi \sqrt{-1} \int_\Sigma f^*(\alpha)}  = \sum_{\rho\in \Irr_\alpha(G)} \left(\frac{|G|}{ \dim \rho}\right)^{-e(\Sigma)},
\end{equation}
where $\Irr_\alpha(G)$ is now the set of isomorphism classes of irreducible projective representations of $G$ whose degree-2 cocycle is specified by $\alpha$.
The use of the pull-back of a cohomology class of the classifying space as the action goes back to \cite{Dijkgraaf:1989pz}. 
This particular formula, appropriately generalized to compact Lie group $G$, first appeared in \cite{Witten:1991we} to the authors' knowledge.
For early rigorous mathematical exposition, see \cite{Freed:1991bn}.

The simplest way to allow $\Sigma$ to be non-orientable is to take $\alpha\in H^2(BG,\Z2)$ instead, since any surface $\Sigma$, non-orientable or otherwise, has a fundamental class $[\Sigma]\in H_2(\Sigma,\Z2)$ and we can pair $f^*(\alpha)$ with it.
The formula now involves the Frobenius-Schur indicator discussed in Lemma~\ref{projFS},
\begin{equation}
\frac{1}{|G|} \sum_{f: \pi_1(\Sigma)\to G } e^{2\pi \sqrt{-1} \int_\Sigma f^*(\alpha)}  = \sum_{\substack{\rho\in \Irr_\alpha(G),\\
 \rho:\text{real}}} S(\rho)^{o(\Sigma)}  \left(  \frac{|G| }{ \dim \rho }\right)^{-e(\Sigma)}
\end{equation}
where $o(\Sigma)=0,1$ is the class $[\Sigma]\in \Omega^\text{unoriented}_{d=2}=\Z2$
and equals $o(\Sigma)=e(\Sigma)$ mod 2.
This formula also already appeared in \cite{Witten:1991we};
for a rigorous mathematical exposition, see \cite{Turaev}.
We also note that for $\Sigma=\mathbb{RP}^1$ and for trivial $\alpha$, the formula reduces to \begin{equation}
 \sum_{g^2=1} 1 = \sum_{\rho\in\Irr(G)} S(\rho) {\dim \rho},
\end{equation}
a formula which appears already in the very paper of Frobenius and Schur \cite{FS}.

More recently, it has become of interest to consider finite group gauge theories on spacetimes with spin structures and variants such as pin$^\pm$ structures.
The motivation came both from developments internal to mathematical physics, see e.g.~\cite{Barrett:2013kza,Novak:2014oca} 
and also from influences from theoretical condensed matter physics, see e.g.~\cite{Kapustin:2014dxa,Gaiotto:2015zta,Wang:2018edf,Guo:2018vij}.

Now, what type of actions $\mathbb{S}$ should we consider?
As was first emphasized in \cite{Freed:2004yc},
natural choices of the topological part of the action should be given by 
invertible quantum field theories; 
for details of invertible quantum field theories, the readers are referred to excellent lecture notes by Freed \cite{FreedLectures}.
For a finite group gauge theory, they are specified by
bordism invariants \begin{equation}
\phi\in \Hom(\Omega^\text{structure}_d(BG),\bR/\bZ)
\end{equation}
where `structure' in the equation stands for `unoriented', `oriented', `spin' etc., 
appropriate for the purpose,
and $d$ is the space-time dimension.
Therefore, we are led to consider the sum \begin{equation}
\frac{1}{|G|} \sum_{f:\pi_1(M_d)\to G} e^{2\pi\sqrt{-1} \phi([f:M_d\to BG])},
\end{equation} 
where $M_d$ is a $d$-dimensional closed and connected manifold equipped with the structure of your choice.

For oriented surfaces with a $G$ bundle, we have \begin{equation}
\Hom(\Omega^\text{oriented}_2(BG),\bR/\bZ)= H^2(BG,\bR/\bZ),
\end{equation} corresponding to the Dijkgraaf-Witten theories we recalled above.
For oriented spin surfaces with a $G$ bundle, we have an equality as sets\begin{equation}
\Hom(\Omega^\text{spin}_2(BG),\bR/\bZ)=H^1(BG,\Z2)  \times H^2(BG,\bR/\bZ),
\end{equation} where the group structure on the right hand side is given by \begin{equation}
(\varphi,\alpha)+(\varphi',\alpha')=(\varphi+\varphi', \alpha + \alpha' + \iota(\varphi\varphi')).
\label{add}
\end{equation}
Here $\iota$ is the homomorphism $\Z2\to \bR/\bZ$ obtained by sending $1$ to $1/2$,
see  \cite{BF3,BF4} where $d=2,3,4$ were treated in parallel.

An element $\phi=(\varphi,\alpha)\in \Hom(\Omega^\text{spin}_2(BG),\bR/\bZ)$
is paired with a $G$-bundle $f:\Sigma\to BG$ on a spin surface in the following manner.
We first recall \cite{Atiyah,Johnson} that a spin structure on a spin surface can be identified with the quadratic refinement $Q: H^1(\Sigma,\Z2)\to \Z2$, i.e.~those functions which satisfy \begin{equation}
Q(a+b)-Q(a)-Q(b)=\int_\Sigma a\cup b
\end{equation} and $Q(0)=0$.
Then we evaluate $\phi$ by the formula \begin{equation}
e^{2\pi \sqrt{-1} \phi([f:\Sigma\to BG])} := (-1)^{Q(f^*(\varphi))} e^{2\pi\sqrt{-1}f^*(\alpha)}.
\end{equation}
The term $\iota(\varphi\varphi')$ in \eqref{add} follows by using the last two formulas.
%We also note that the addition formula \eqref{add} is the same as the one given in \eqref{grouptensor}.

The generalization of the formula \eqref{Med} is given by 
 \begin{multline}
 \qquad\qquad
\frac{1}{|G|} \sum_{f: \pi_1(\Sigma)\to G } (-1)^{Q(f^*(\varphi))} e^{2\pi \sqrt{-1} \int_\Sigma f^*(\alpha)}  = \\
\qquad\qquad\qquad\qquad\qquad
\sum_{\rho\in \Irr(\bC[G]_{\varphi,\alpha})} (-1)^{\Arf(\Sigma) q(\rho)} \left(  \frac{ |G| }{ \qdim \rho } \right)^{-e(\Sigma)},
 \qquad\qquad
\end{multline}
where the sum on the right hand side is now over the irreducible supermodules $\rho$ of $\bC[G]_{\varphi,\alpha}$ discussed above,
$q(\rho)\in \BW(\bC)=\{0,1\}$ is the type of the supermodule $\rho$,
$\qdim(\rho)=\dim(\rho)/\sqrt{2}^{q(\rho)}$,
and the Arf invariant $\Arf(\Sigma)=\Z2$ is the class $[\Sigma]\in \Omega_2^\text{spin}=\{0,1\}$ of the spin surface in the bordism group.
This formula was derived in \cite{Gunningham}; 
the presentation there was mainly for the case $G=((\Z2)^n)\rtimes S_n$ with a specific choice of $\varphi$ and $\alpha$, for which $\bC[G]_{\varphi,\alpha}$ is known as the Sergeev algebra, but the discussion there easily generalizes. 

Our super Frobenius-Schur indicator plays a role in the case of unoriented surfaces $\Sigma$ equipped  with a pin$^-$ structure.
In this case we have an equality of sets\cite{BFq}
\begin{equation}
\Hom(\Omega^\text{pin$^-$}_2(BG),\bR/\bZ)=H^1(BG,\Z2)  \times H^2(BG,\Z2) 
\end{equation}
where the addition formula on the right hand side is
\begin{equation}
(\varphi,\alpha)+(\varphi',\alpha')=(\varphi+\varphi', \alpha + \alpha' + \varphi\varphi').
\end{equation}
Note that this has the same form as the formula we saw in Lemma~\ref{grouptensor}.

To evaluate $\phi=(\varphi,\alpha)\in \Hom(\Omega^\text{pin$^-$}_2(BG),\bR/\bZ)$,
we now use the fact that a pin$^-$ structure on a surface is given by a $\Z4$-valued quadratic refinement $\cQ:H^1(\Sigma,\Z2)\to \Z4$ satisfying \begin{equation}
\cQ(a+b)-\cQ(a)-\cQ(b)=2\int_\Sigma a\cup b
\end{equation} and $\cQ(0)=0$,
see \cite{KirbyTaylor} for this and other basic properties of  pin$^-$ structures.
Then the evaluation is given by
\begin{equation}
e^{2\pi \sqrt{-1} \phi([f:\Sigma\to BG])} := \sqrt{-1}^{\cQ(f^*(\varphi))} e^{2\pi\sqrt{-1}f^*(\alpha)}.
\end{equation}

Finite-group gauge theories on pin$^-$ surfaces for the case when the twisted group superalgebra  is the Clifford algebra $\Cl(n)$ we saw in Lemma~\ref{clifford0} were studied e.g.~in \cite{DebrayGunningham,Turzillo:2018ynq,Kobayashi:2019xxg}.
Again the discussion there can be readily generalized, given our super Frobenius-Schur indicator.
The final result when $\Sigma$ is non-orientable is
 \begin{multline}
 \qquad\qquad
\frac{1}{|G|} \sum_{f: \pi_1(\Sigma)\to G } \sqrt{-1}^{\cQ(f^*(\varphi))} e^{2\pi \sqrt{-1} \int_\Sigma f^*(\alpha)}  = \\
\sum_{\substack{\rho\in \Irr(\bC[G]_{\varphi,\alpha}),\\
\rho:\text{real}}}
 \cS(\rho)^{\ABK(\Sigma)} \left(  \frac{ |G| }{ \qdim \rho }\right)^{-e(\Sigma)},
 \qquad\qquad
\end{multline}
where the Arf-Brown-Kervaire invariant $\ABK(\Sigma) \in \{0,1,\ldots,7\}$ of the pin$^-$ surface is the class $[\Sigma] \in \Omega^\text{pin$^-$}_{d=2}=\Z8$.
The defining equation of the super Frobenius-Schur indicator in the statement of Theorem~\ref{main} is the literal path integral expression of the partition function of the gauge theory on the M\"obius strip, when the state on the boundary is specified by $\rho$.

\bibliographystyle{ytamsalpha}
%\baselineskip=.95\baselineskip
\bibliography{ref}

\end{document}